\documentclass{amsart}
\usepackage{amsmath,amsfonts,amsthm,amssymb}
\usepackage{cite}
\usepackage[utf8]{inputenc}
\usepackage{amsmath}
\usepackage{amsfonts}
\usepackage{amssymb}
\usepackage{amsthm}
\usepackage{mathalfa}
\usepackage{graphicx}
\usepackage[colorinlistoftodos]{todonotes}
\usepackage{enumitem}
\usepackage{xfrac}
\usepackage[margin=1in]{geometry}
\usepackage{mathtools}
\usepackage[all]{xy}
\usepackage{bm}

\newif\ifthesis
\thesistrue

\usepackage[colorlinks=false]{hyperref}

\newcommand{\ZZ}{\mathbb{Z}}
\newcommand{\QQ}{\mathbb{Q}}

\newcommand{\lam}{\lambda}

\newcommand{\FF}{\mathbb{F}}

\newcommand{\lct}{\mathrm{lct}\,}

\newcommand{\then}{\Rightarrow}

\newcommand{\calP}{\mathcal{P}}

\newcommand{\alp}{\alpha}

\usepackage{aliascnt}

\theoremstyle{definition}
\newtheorem{theorem}{Theorem}[section]

\newaliascnt{lemct}{theorem}

\newtheorem{lem}[lemct]{Lemma}
\aliascntresetthe{lemct}

\newaliascnt{propct}{theorem}
\newtheorem{prop}[propct]{Proposition}
\aliascntresetthe{propct}

\newaliascnt{corct}{theorem}
\newtheorem{cor}[corct]{Corollary}
\aliascntresetthe{corct}

\newaliascnt{exmct}{theorem}
\newtheorem{exm}[exmct]{Example}
\aliascntresetthe{exmct}

\newaliascnt{exerct}{theorem}

\aliascntresetthe{exerct}

\newaliascnt{discct}{theorem}
\newtheorem{disc}[discct]{Discussion}
\aliascntresetthe{discct}

\newaliascnt{rmkct}{theorem}
\newtheorem{rmk}[rmkct]{Remark}
\aliascntresetthe{rmkct}

\newaliascnt{clmct}{theorem}

\aliascntresetthe{clmct}

\newaliascnt{qstct}{theorem}
\newtheorem{qst}[qstct]{Question}
\aliascntresetthe{qstct}

\newaliascnt{deffct}{theorem}
\newtheorem{deff}[deffct]{Definition}
\aliascntresetthe{deffct}

\usepackage{thmtools}
\usepackage{thm-restate}

\newcommand{\st}{such that }

\newcommand{\twlogd}[1]{without loss of generality{#1}}

\newcommand{\tiff}{if and only if }

\usepackage{marginnote}

\newcommand{\bref}[1]{\textbf{\autoref{#1}}}
\newcommand{\pref}[1]{(\ref{#1})}

\newcommand{\PP}{\mathbb P}

\newif\ifdebug
\debugfalse

\usepackage{verbatim}
\usepackage{courier}

\newcommand{\labelt}[1]{\label{#1}\ifdebug \texttt{\smaller{#1     }} \fi}

\numberwithin{equation}{theorem}

\title{Legendre Polynomials Roots and the $F$-Pure Threshold of bivariate Forms}  

\author{Gilad Pagi}
\thanks{The author  acknowledges the partial financial support of NSF grant DMS-0943832.}

\begin{document}

\maketitle
\begin{abstract}
We provide a direct computation of the $F$-pure threshold of degree four homogeneous polynomials in two variables and, more generally, of certain homogeneous polynomials with four distinct roots. The computation depends on whether the cross ratio of the roots satisfies a specific M\"{o}bius transformation of a Legendre polynomial. We then make a connection between a long lasting open question, involving the relationship between the $F$-pure and the log canonical threshold, and roots of Legendre polynomials over $\FF_p$. 
\end{abstract}
\section{Introduction}
In this note, we provide an elementary computation of the $F$-pure threshold of the homogeneous defining equation of a family of subschemes of $\PP^1$ supported at four points. Our formula depends on whether the cross-ratio of these four points satisfies a certain \textit{Deuring} Polynomial; a Deuring polynomial is a M\"{o}bius transformation of a \textit{Legendre} Polynomial of the same degree (See \bref{deuDef}).

Let $K$ denote a field of prime characteristic $p$ and let $R=K[x_1,...,x_t]$. Fix any polynomial $f\in R$. By the $F$-pure threshold (at the origin) we mean:
 \begin{equation}\labelt{FTdef}
 FT(f):=\sup\left\{ \frac{N}{p^e}\mid  N,e\in  \ZZ_{>0},  f^N\not\in (x_1^{p^e},...,x_t^{p^e})R   \right\}.
 \end{equation} This definition appeared in \cite{BMS09} (the original tight-closure formulation is stated in \cite{TW04}). The $F$-pure threshold is a characteristic $p$ analog of the log canonical threshold of a complex singularity (as defined in \cite{KOL97}). A famous open conjecture, stated in \bref{openQ}, relates these two thresholds; interestingly, our work reduces this conjecture, for a certain family of bivariate forms, to understanding roots of Legendre polynomials over $\FF_p$.

Our first goal is to compute the $F$-pure threshold of a bivariate homogeneous polynomial of degree four. Because the case of multiple roots is easy (see \bref{restForms}), our main result treats the case where the roots are all distinct:
\begin{theorem}\labelt{mainThm} Let $K$ be a field of prime characteristic $p$. Consider a degree four homogeneous polynomial $f\in K[x,y]$, with distinct roots over $\PP_{\overline{K}}^1$. After fixing an order of the roots, let $a\in \overline{K}$ be their cross-ratio. Denote $n_1 = \frac{p-1}{2}$, and let $H\{n_1\}(\lam)\in K[\lam]$ be the Deuring polynomial (defined in \bref{deuDef}) of degree $n_1$. Then 
$$
FT(f) = \left\{
\begin{array}{ll}
\frac{1}{2}& \text{ if } p=2 \text{ or if both } p>2 \text{ and } H\{n_1\}(a) \neq 0 \\
\frac{1}{2}\left(1-\frac{1}{p}\right)& \text{ if } p>2 \text{ and } H\{n_1\}(a) = 0
\end{array}\right.
$$
\end{theorem}
It is intriguing that the value of the $F$-pure threshold depends on whether the cross-ratio satisfies some (M\"{o}bius transformation of) Legendre polynomial.
The technique we use in the proof relies on the properties of the {Deuring Polynomials} as presented in \cite{SELF}. While some of these properties can be deduced from known facts about Legendre polynomials, we include straightforward algebraic proofs (or cite some from \cite{SELF})
in order to be self-contained.

We generalize \bref{mainThm} to certain higher degree polynomials:
\begin{restatable}{theorem}{thmTwo}\labelt{thmTwoL} Let $K$ be a field of prime characteristic $p$. Let $c,b\in \ZZ_{>0}$ with $p \equiv 1 \pmod{b+c}$. Let $f\in K[x,y]$ be a homogeneous polynomial of degree $2b+2c$ with exactly four distinct roots over $\PP_{\overline{K}}^1$, where the multiplicities are $b,b,c,c$ after fixing an order. Let $a$ be the their cross-ratio. Denote $n=\frac{c}{c+b}(p-1)$. Then
 $$
FT(f) = \left\{
\begin{array}{ll}
\frac{1}{b+c}& \text{ if } H\{n\}(a) \neq 0 \\
\frac{1}{b+c}\left(1-\frac{1}{p}\right)& \text{ if } H\{n\}(a) = 0
\end{array}\right.
$$
\end{restatable}

\begin{disc}\labelt{openQ} Our theorem connects a well known open question regarding the relationship between $F$-pure and log canonical threshold to a seemingly unrelated question of roots of Legendre polynomials over $\FF_p$. Recall the open question. Consider a polynomial $f$ with integer coefficients\footnote{The same can be done for a complex polynomial, but further technical steps are needed.}. On the one hand, we can consider $f$ as a complex polynomial, and compute the log canonical threshold of $f$, denoted $\lct(f)$. On the other hand we can compute $FT(f_p)$ repeatedly for \textit{each $p$}, where $f_p$ is the natural image of $f$ over $\ZZ_p$. Let $\calP$ be the set of all primes $p$ \st $FT(F_p)=\lct(f)$. A decades old open conjecture predicts that $\calP$ is of infinite cardinality\footnote{The conjecture as stated appears in \cite{MTW05}, but it roots dates back to the work of the Japanese school of tight closure (see \cite{HW02}). Surveys and other formulation can be found in \cite{SMI97}, \cite{EM06}. See some progress in \cite{HerLog}.}. It is well known that $\lim_{p\to\infty} FT(F_p)=\lct(f)$.\footnote{This observation is the culmination of a series of papers, going back to \cite{HH90}, \cite{Smi00}, \cite{H01}, \cite{HW02}, \cite{HY03}, \cite{T04}, \cite{HT04}, \cite{TW04}, \cite{MTW05}. For a gentle introduction see \cite{KSbasic}.}

We now point out how this open question relates to Legendre polynomials for the case of the family of polynomials in \bref{thmTwoL}. Let $f$ be a polynomial as in the theorem. One can compute that $\lct(f)=\frac{1}{b+c}$. In order to verify the conjecture for this specific family of polynomials, one should prove that there are infinitely many $p$'s \st the cross ratio of $f_p$ is not a root of $H\left\{\frac{c}{b+c}(p-1)\right\}$ over $\overline{F_p}$. For example, here is a precise formulation of our statement is the simplest case. 
\begin{qst}
Suppose $f = x^by^b(x+y)^c(x+ay)^c\in \ZZ[x,y]$. Denote $$\calP = \left\{\text{all primes } p\,\, \left| \,\, p\equiv 1 \pmod{b+c} \,\,\,\text{   and   }H\left\{\frac{c}{b+c}(p-1)\right\}(a)\not\equiv 0 \pmod{p} \right.\right\}.$$
Is it true that the cardinality of $\calP$ is infinite?
\end{qst}

This may be very difficult, and is related to deep theorems in number theory. For example, the case where $b=c=1$ is already known as it is equivalent to the fact that there are infinitely many $p$'s \st an elliptic curve is ordinary (see \cite{SELF}). Further evidence that the conjecture is  connected to ordinarity is explored in \cite{MS11}
\end{disc}

In addition, the $F$-pure threshold computation in \bref{mainThm} provides an immediate alternative proof for the known corollary regarding properties of the roots of Legendre polynomials mod $p$:
\begin{cor}
Fix a prime $p>2$, a field $K$ of characteristic $p$ and let $n=\frac{p-1}{2}$. If $b \in K-\{\pm 1\}$ is a root of the Legendre polynomial of degree $n$, $P_n(x) \in K[x]$, then these are roots as well:
$$
\pm b, \pm \frac{3+b}{-1+b}, \pm \frac{3-b}{1+b}.
$$
\end{cor} See \textbf{Section \ref{corSec}}.
\subsection*{Acknowledgments}
This article is part of my Ph.D. thesis (\cite{pagithesis}), which was written under the direction of Karen Smith of University of Michigan. I would like to thank Prof. Smith for many useful discussions.

\section{Roots of Deuring Polynomials In Prime Characteristic}

A crucial part of computing $FT(f)$ is identifying when coefficients of monomials of $f^N$ vanish. We later observe (\bref{hassePolyGeneral}) that one of these coefficient is no other than the Deuring polynomial (\bref{deuDef}) evaluated at the cross-ratio. Therefore, we turn to investigate roots of Deuring polynomials in prime characteristic. 

\begin{deff}\labelt{deuDef}
 For an integer $n>0$, define the following polynomial in $\ZZ[\lam]$:
 $$
 H\left\{n\right\}(\lam): = \sum_{i=0}^n {n\choose i}^2\lam^i
 $$
 \end{deff}
 Following \cite{MOR}, we call $H\{n\}(\lam)$ the \textit{Deuring Polynomial}\footnote{Arguably it first appeared in \cite{DUR}} of degree $n$; it can be equivalently defined using the Legendre polynomial of degree $n$, $P_n(x)$:
 \begin{equation}\labelt{DeurToLeg}
 H\{n\}(\lam)=(1-\lam)^nP_n\left(\frac{1+\lambda}{1-\lambda}\right).
 \end{equation}
 (see \cite{SELF} for more details.) When the indeterminate $\lam$ is understood from the context we omit it and write $H\{n\}$. We often abuse notation and write $H\{n\}\in \FF_p[\lam]$ for the natural image of this integer polynomial mod $p$.

We shall investigate the roots of $H\{n\}$ in characteristics $p$. The following two lemmas are proven in \cite{SELF}:

\begin{lem}\labelt{pMinusOneHass}
Let $p$ be a prime. Then $H\{p-1\}\in \FF_p[\lam]$ is $(\lam-1)^{p-1}$.
\end{lem}

\begin{lem}[\textbf{Schur's Congruence}]\labelt{hasseLucas} Fix a prime $p$. 
Let $H\{n\}\in \FF_p[\lam]$. Write the $p$-expansion of $n$:
$$
n=b_0p^0+b_1p^1+...+b_ep^e.
$$

Then
$$
H\{n\} = H\{b_0\}^{1}H\{b_1\}^{p^1}H\{b_2\}^{p^2}\cdots H\{b_e\}^{p^e}
$$
\end{lem}

\begin{exm}\labelt{neHasse} In characteristic $p$, where $p$ is an odd prime:
$$
H\left\{\frac{p^e-1}{2}\right\} = H\left\{\frac{p-1}{2}\right\}^{1+p+...+p^{e-1}}
$$
\end{exm}

For the computation of our main theorems, we need the following properties of the roots of $H\{n\}$ in characteristic $p$:
\begin{lem}\labelt{Hsimple} Fix a prime $p$, and an integer $0\leq n < p/2$. Let $K$ be a field of characteristic $p$. Then $H\{n\}\in K[\lam]$ has no repeated roots. Further, $\lam=0,1$ are not roots of $H\{n\}$. Moreover, if  $0< n < p/2$ then $H\{n\}$ and $H\{n-1\}$ share no roots.
\end{lem}

The above lemma follows from the orthogonality of Legendre polynomials and the recursive relation between them (see the author's Ph.D. thesis for more details). If the reader would like to avoid analytic techniques, we provide a purely algebraic proof of \bref{Hsimple} which reveals interesting properties of Deuring polynomials. The rest of this section is dedicated for that goal. Alternatively, the reader may skip directly to the computation of \bref{thmTwoLSimp} in the next section. 

\bigskip

In \cite{SELF}, we prove:
\begin{lem}\labelt{pascalConnection}Fix an integer $n\geq 0$.
Let $F=F(\lam)\in \QQ[\lam]$ be the formal antiderivative of the polynomial $H\{n-1\}(\lam)$ with constant coefficient $0$. We denote $H\{n-1\} = F'$. Then
$$
 (1-\lam)F' + 2nF = H\{n\}.
$$ Note that this equality holds characteristic 0 and thus in all positive characteristics $n<p$.
\end{lem}

We next develop differential equations for $H\{n\}$ it formal antiderivative that will help us to investigate their roots.
Once we do that, we can use the following lemma to deduce properties of their roots:

\begin{lem}\labelt{repRootsDELem} Let $K$ be a field of prime characteristic $p$. Let $F\in K[\lam]$  of degree $d<p$ and denote $F',F''$ as its first and second derivative, respectively. Suppose that $F$ satisfies a differential equation of the form 
\begin{equation}\labelt{repRootsDE}
\lam(\lam-1)F'' + a\lam F' + bF' + c F = 0, \,\,\, a,b,c \in K.
\end{equation}
Then the only possible repeating roots of $F$ are $\lam= 0$ and $\lam = 1$. 
\end{lem}
\begin{proof}
 Suppose $\alp$ is a root of $F$ of multiplicity $r\geq 2$. Since $\deg F = d < p$, then $r<p$. So write
\[
\begin{array}{llcl}
F = &g_1(\lam) \cdot (\lam-\alp)^r&\text{ where }&g_1(\alp) \neq 0 ,\\
F' =&g_2(\lam) \cdot  (\lam-\alp)^{r-1}&\text{ where }&g_2(\alp)\neq 0 ,\\
F'' =&g_3(\lam) \cdot  (\lam-\alp)^{r-2}&\text{ where }&g_3(\alp) \neq 0 .\\
\end{array}
\]
Plug the above expression in (\ref{repRootsDE}) and divide by $(\lam-\alp)^{r-2}$ to get
$$
\lam(\lam-1)g_3 + (a\lam + b)(\lambda-\alpha) g_2 + c (\lam-\alp)^2g_1 = 0 .
$$Plugging in $\lam=\alp$ gives:
$$
\alp(\alp-1)g_3(\alp)= 0
$$ We get:
$$
\alp(\alp-1)= 0 \then \alp=0,1 
$$i.e. the only possible repeated roots of $F$ are $\alp=0$ or $\alp=1$.
\end{proof}

\begin{prop}\labelt{Fsimple} Fix $n\geq 0$. Let $F=F(\lam)\in \QQ[\lam]$ be the formal antiderivative of $H\{n\}\in \ZZ[\lambda]$ with constant coefficient 0. Then $F$ satisfies: 
\begin{equation}\labelt{diffOppF}
\lam(\lam-1)F'' -(1+2n)\lam F' + (n+1)^2F=0
\end{equation} Further, if $K$ is a field of prime characteristic $p$ and $0\leq n <p/2$, then $F$ has a natural image in $K[\lam]$ and it only has simple roots over $K$. 
\end{prop}
\begin{proof} There is a constructive proof for the validity of \pref{diffOppF} in \cite[Lemma III.18]{pagithesis}. One can also simply plug in:
$$
F(\lam) = \sum_{i=0}^n {n \choose i}^2\frac{1}{i+1}\lam^{i+1}
$$ and observe that $F$ satisfies the differential equation (\ref{diffOppF}) in characteristic $0$ and thus in every characteristics in which $F$ can be defined. A sufficient condition is $n+1<p$ since in this case we can invert all the power of $H\{n\}$. Let $K$ be field of prime characteristic $p$ with $0\leq n < p/2$. Since for all primes $p/2 \leq p-1$, the condition $n<p/2$ guarantees that we can define $F$ in $K[\lam]$. Using \bref{repRootsDELem}, the above differential equation shows that the only possible repeating roots of $F$ is 0 and 1. However, they are not roots of $H\{n\}=F'$ as proven independently later in \bref{Hsimple2}. 
\end{proof}

\begin{lem}\labelt{DiffOppH} Let $n \geq 0$ be an integer and denote $H=H\{n\}\in \ZZ[\lam]$. Then $H$ satisfies the following differential equation:
\begin{equation}\labelt{diffOpHn}
\lam(\lam-1)H'' + (\lam(1-2n)-1)H' +n^2H=0.
\end{equation}
\end{lem}
\begin{proof}
Simply take the derivative of \pref{diffOppF}. 
\end{proof}

\begin{rmk} 
For example, set $n=\frac{p-1}{2}$ for an odd prime $p$, and multiply by $4$ in order to clear denominators. We get:
$$
4\lam(\lam-1)H'' + 4(\lam(2-p)-1)H' +(p-1)^2H=0.
$$
Over $\FF_p$, this equations becomes:
$$
4\lam(\lam-1)H'' + 4(2\lam-1)H' +H=0,
$$
which is identical to the \textit{Picard-Fuchs} operator (see \cite[Remark 4.2]{SILV}). In many cases $n$ is a polynomial in $p$ with rational coefficients, say $n=g(p)$. So when working in $\FF_p$, one can replace $n$ by $g(p)$, clear denominators and get a differential operator over $\FF_p$ which does not depend on $n$. 
\end{rmk}

Now we conclude the first part of \bref{Hsimple}:
\begin{cor}\labelt{Hsimple2} Fix a prime $p$, and an integer $0\leq n < p/2$. Let $K$ be a field of characteristic $p$. Then $H\{n\}\in K[\lam]$ has no repeated roots. Further, $\lam=0,1$ are not roots of $H\{n\}$.
\end{cor}
\begin{proof}
Let $H=H\{n\}$. Combining \bref{DiffOppH} and \bref{repRootsDELem} shows that the only possible repeating roots of $H$ are $0$ and $1$. However, $H(0)=1$. Moreover, the following combinatorial identity (which holds over $\ZZ$) shows:
$$
H\{n\}(1)=\sum_0 ^{n} {n \choose i}^2 = {2n \choose n}.
$$
This is non-zero because $2n<p$, thus $\lam=1$ is not a root of $H$ as well.
\end{proof}

Now we conclude the last part of \bref{Hsimple}:
\begin{cor}\labelt{FsimpleNoSharedRoots}
Fix an integer $n\geq 1$ and a prime $p$ \st $n<p/2$. Let $K$ be a field of characteristic $p$. Then $H\{n\}$ and $H\{n-1\}$ share no roots.
\end{cor}
\begin{proof}
Let $F$ be the formal antiderivative of $H\{n-1\}$ with constant coefficient 0.
Consider the ideal $I=(H\{n\},H\{n-1\})$ in $K[\lam]$. From \bref{pascalConnection} we have:
$$
I=(H\{n\},H\{n-1\}) = ((1-\lam) F' + 2n F,F') = (2n F, F') = (F,F'),
$$where the last inequality holds since $2n$ is a unit in $\FF_p$ and thus in $K$. 
Therefore, $I$ is the unit ideal \tiff $F$ only has simple roots, which is the result in \bref{Fsimple}.
\end{proof}

\section{Computation of the $F$-pure threshold}

We start this section with two useful observations for computing $FT(f)$. Let $K$ be a field. A polynomial $f\in K[x_1,...,x_t]$ is a linear combination of monomials over $K$. Denote the monomial $x_1^{\mu_1}\cdots x_t^{\mu_t}$ by $\textbf{x}^{\bm{\mu}}$ where $\bm\mu$ is the multiexponent $[\mu_1,...,\mu_t]$. Similarly, for $s$ scalars in $K$, $b_1,...,b_s$, we denote $\textbf{b}=[b_1,...,b_s]$. Now, let $\textbf{x}^{\bm{\mu}_1},...,\textbf{x}^{\bm{\mu}_s}$ be the monomials of $f$. Using the usual meaning of dot product we have:
$$
f = \textbf{b} \cdot [\textbf{x}^{\bm{\mu}_1}, ..., \textbf{x}^{\bm{\mu}_s}] = b_1\textbf{x}^{\bm{\mu}_1} + ...  + b_s\textbf{x}^{\bm{\mu}_s}.
$$ 
For a multi-exponent $\textbf{k}=[k_1,...,k_t]$ we denote $\max \textbf{k}$ as the maximal power in the multiexponent $\textbf{k}$, i.e. 
$$
\max \textbf{k}=\max [k_1,...,k_t] = \max_{1\leq i \leq t} k_i.
$$
Using this notation, we have the following straightforward way to produce upper and lower bounds for $FT(f)$:
\begin{lem}\labelt{upLowBound}
Let $R=K[x_1,...,x_t]$ where $K$ is a field of prime characteristics $p$, and let $f\in R$. Let $N$ be a positive integer. Raise $f$ to the power of $N$ and collect all monomials, so that:
\begin{equation} \labelt{eqFtoNcollected}
f^N = \sum_{\text{distinct multi-exponents } \textbf{k}} c_{\textbf{k}} \textbf{x}^{\textbf{k}}.
\end{equation}
Note that all but finitely many $c_{\textbf{k}}$'s are 0. 
 Fix $e\in \ZZ_{\geq 0}$ and consider $\frac{N}{p^e}$. Then:
\begin{enumerate}
\item $ \frac{N}{p^e}<FT(f) \iff \exists \textbf{k}$ \st $c_{\textbf{k}}\neq 0$ and $\max \textbf{k} < p^e$.
\item $FT(f)\leq \frac{N}{p^e} \iff \forall \textbf{k}$, either $c_{\textbf{k}}= 0$ or $\max \textbf{k} \geq p^e$.
\end{enumerate}
\end{lem}
\begin{proof}
This is immediate from the definition, specifically \pref{FTdef}, and from \cite[Prop 3.26]{KSbasic} which implies that for any $\frac{N}{p^e}\in [0,1]$, $$f^N \not\in (x_1^{p^e},...,x_t^{p^e})R \iff \frac{N}{p^e} < FT(f).$$
\end{proof}

\begin{lem}\labelt{betaBoundHomog} Let $f$ be a homogeneous polynomial of degree $d$ in $t$ variables. Let $\textbf{x}^{\textbf{k}}$ be a monomial in $f^N$ with a non-zero coefficient. Denote $\textbf{k}=[k_1,...,k_t]$. Then $k_1+...+k_t=dN$. Moreover, $\max \textbf{k}\geq Nd/t$ and if $\max \textbf{k} = Nd/t$ then $\textbf{k} = [Nd/t,Nd/t,...,Nd/t]$.
\end{lem}
\begin{proof}
 The first statement is immediate since any monomial of $f^N$ is of degree $dN$. Ergo, we cannot have that all $t$ entries of $\textbf{k}$ are less than $Nd/t$. Lastly, if $\max \textbf{k} = Nd/t$ but another power is less, then $k_1+...+k_t$ is less than $Nd$. 
\end{proof}

We can now focus on the polynomials appearing in our main theorems. Let $f$ be a bivariate degree four homogeneous polynomial. We would like to reduce the problem of computing $FT(f)$ of this quite general polynomial to a problem of computing the $F$-pure Threshold of a more ``canonical" polynomial.

\begin{prop}\labelt{fiveForms} Let $f \in K[x,y]$ be a degree four homogeneous polynomial over a field $K$ of characteristic $p$. Then $FT(f)$ is identical to the $F$-pure threshold of one of the following polynomials:
\begin{equation}\labelt{fiveFormsEq}
x^4, x^3y, x^2y^2, x^2y(x+y), xy(x+y)(x+a y) \text{ with } a \in \overline{K}-\{0,1\}.
\end{equation}

\end{prop}
\begin{proof}
$FT(f)$ is preserved under base change, scalar multiplication and linear change of variables. Thus, \twlogd, let $K$ be algebraically closed, over which $f$ factors as a product linear terms. Now change variables to obtains one of the five forms in \pref{fiveFormsEq}, and suffices to compute $FT(f)$ for each of these cases.
\end{proof}

We are interested in the last form, since the $F$-pure threshold can be computed easily in the rest of the cases. For completeness, we comment about them in \bref{restForms}.

The next lemma shows that understanding the Deuring polynomial $H\{n\}$ is crucial for the discussion.

\begin{lem}[\textbf{Main Technical Lemma}]\labelt{hassePolyGeneral}
 Let $f_\lam=(x+y)(x+\lam y)$ and let $N$ be a positive integer. Then the coefficient of $x^{N}y^{N}$ in $f_\lam^N$ is $H\{N\}(\lam)$.
\end{lem}
\begin{proof}
Notice:
$$
f_\lam^N = (x+y)^{N}(x+\lam y)^{N} = \left(\sum_{i=0}^{N} {N \choose i} x^{i}y^{N-i}\right) \left(\sum_{j=0}^{N} {N\choose j} (\lam)^{j}x^{N-j}y^{j}\right).
$$
For the coefficient of $x^{N}y^{N}$ we need to set $i=j$, so we end up with:
$$
\sum_{i=0}^{N} {N \choose i}^2 \lam^{i} = H\{N\}.
$$
As required.
\end{proof}

Consider the statements in \bref{mainThm} and \bref{thmTwoL}, and let us reduce them to a more computationally friendly theorem. Since the $F$-pure threshold is invariant under base change and linear change of variables, we can assume $K=\overline{K}$. In light of \bref{fiveForms}, the polynomials in these statements can be fixed to have the form $xy(x+y)(x+a y)$ or $x^by^b(x+y)^c(x+a y)^c$ respectively, as it is easy to see that $a$ is the cross-ratio of the roots once we fix an order and that $a$ cannot be $0,1$ or $\infty$ since the roots are all distinct. The equivalent statement we get is:

\begin{restatable}{theorem}{thmTwoSimp}\labelt{thmTwoLSimp} Let $K$ be a field of prime characteristic $p$. Let $c,b\in \ZZ_{>0}$ with $p \equiv 1 \pmod{b+c}$. Fix $f\in K[x,y]$ of the form:
\begin{equation}
f_a = x^by^b(x+y)^c(x+a y)^c,\, a\in K-\{0,1\},\, 
\end{equation}
Denote $n=\frac{c}{c+b}(p-1)$ and let $H\{n\}(\lam)\in K[\lam]$ be the Deuring polynomial of degree $n$. Then
 $$
FT(f_a) = \left\{
\begin{array}{ll}
\frac{1}{b+c}& \text{ if } H\{n\}(a) \neq 0 \\
\frac{1}{b+c}\left(1-\frac{1}{p}\right)& \text{ if } H\{n\}(a) = 0
\end{array}\right.
$$
\end{restatable}
As long as $p\neq 2$, \bref{mainThm} is a special case of \bref{thmTwoLSimp} in which $b=c=1$. Note also that the $p\neq 2, b=c=1$ scenario is also provable by applying \cite[Theorem 3.5]{HNWZ16} with $a=L=1,b=2$; however the computation is not direct. The proof of the general \bref{thmTwoLSimp} follows next where the $p=2$ special case is proven right after. 

\begin{proof}
The key observation is that for a positive integer $N$, we use \bref{hassePolyGeneral} and \bref{betaBoundHomog} to deduce:
\begin{equation}\labelt{keyObPlcG}
f_a^N = x^{bN}y^{bN}((x+y)(x+a y))^{cN} = x^{(b+c)N}y^{(b+c)N}H\{cN\}(a) + I 
\end{equation}
where $I$ is an element in the ideal $(x^{(b+c)N+1},y^{(b+c)N+1})$.
Let us prove that $1/(b+c)$ is an upper bound. Fix an integer $e>0$ and set $N=\frac{1}{b+c}(p^e-1+p-1)$. From \pref{keyObPlcG}, combined with \bref{upLowBound}, we get the $N/p^e=\frac{1}{b+c}\frac{p^e+p-2}{p^e}$ is an upper bound. Taking $e\to \infty$, we get that $FT(f_a)\leq \frac{1}{b+c}$ as required. 

In the case that $H\{n\}(a)\neq 0$, we wish to show that $\frac{1}{b+c}$ is also a lower bound. With $e>0$ and $N=\frac{1}{b+c}(p^e-1)$, the coefficient of $x^{(b+c)N}y^{(b+c)N}$ in $f_a^N$ is $H\{cN\}(a)$. Since $(b+c)N = p^e-1<p^e$, showing that $H\{cN\}(a)\neq 0$ would establish $N/p^e = \frac{1}{b+c}\frac{p^e-1}{p^e}$ as a lower bound for any $e>0$, and thus $\frac{1}{b+c} \leq FT(f_a)$. Let us compute the $p$-expansion of $cN$:
$$
cN = \frac{c}{b+c}(p^e-1) = \frac{c(p-1)}{b+c} +\frac{c(p-1)}{b+c}p + ... +\frac{c(p-1)}{b+c}p^{e-1} 
$$
Note that $n=\frac{c(p-1)}{b+c}$ is between $0$ and $p-1$. Ergo, by \bref{hasseLucas}
\begin{equation}\labelt{eqForq}
H\{cN\}(a) = H\{\frac{c}{b+c}(p^e-1)\}(a)= \left(H\{n\}(a)\right)^{g} \neq 0
\end{equation} where $g$ is the resulting positive integer exponent (its exact value is not important). In the case that $H\{n\}(a)= 0$, we would like to show that $FT(f_a) = \frac{1}{b+c}\left(1-\frac{1}{p}\right)$. To establish that value as an upper bound, consider again  $N=\frac{1}{b+c}(p^e-1)$. From (\ref{keyObPlcG}) and (\ref{eqForq}) we see that $H\{cN\}(a)=0$ and thus $f_a^N \in (x^{p^e},y^{p^e})$, making $N/p^e$ an upper bound. Plug in $e=1$ to see that $\frac{1}{b+c}\left(1-\frac{1}{p}\right)$ is indeed an upper bound.

As for a lower bound, first apply an appropriate change of coordinates, if needed, to ensure that $b\geq c$. Now, recall \bref{Hsimple}. Note that $n<p/2$ and $H\{n\}(a)= 0$, thus $H\{n-1\}(a)\neq 0$. Since $c<p$, and $p$ is a prime, there is a power of $p$ that is congruent to $1$ mod $c$. Denote it as $p^{d}$. For an integer $m$, $p^{md} \equiv 1 \pmod c$ and thus we define:
$$
\ell(m):=(p-n)p^{md-1} \equiv 1 \pmod c,
$$ because $c$ divides $n$. 

Now, consider the integer
$$
N' = (p-1)p^0+(p-1)p^1 +...+(p-1)p^{e-2}+(n-1)p^{e-1} = p^{e-1}-1+(n-1)p^{e-1}=np^{e-1}-1
$$ for $e\gg 1$. The digits of the $p$ expansion are $(p-1)$ and $(n-1)$. We cannot just yet use $N'$ as $cN$ since it is not necessarily divisible by $c$. By subtracting $\ell(1)$ from $N'$ we are making the $p^{d-1}$ digit become $(n-1)$ instead of $(p-1)$. Then we shall do the same for the $p^{2d-1}$ digit, the $p^{3d-1}$ digit and so on, through the $p^{(c-1)d-1}$ digit. Now we get an integer divisible by $c$ and we can define $N$:
$$
cN = N' - \ell(1) - ... - \ell(c-1) = np^{e-1} - 1 - \ell(1) - ... - \ell(c-1),
$$
$$
N = \frac{1}{b+c}(p-1)p^{e-1} - L,
$$ were $L$ is some integer constant, not dependent on $e$. We are about to show that $N/p^e$ is a lower bound for an arbitrary large $e$, which will complete the proof. 
Notice that $(b+c)N=(p-1)p^{e-1}-(b+c)L<p^e$, while the coefficient of $x^{(b+c)N}y^{(b+c)N}$ in $f^N$ is $H\{cN\}$. We carefully crafted $cN$ to have a $p$ expansion containing only digits of $(p-1)$ or $(n-1)$. Using \bref{hasseLucas}, we have:
$$
H\{cN\} = H\{p-1\}^{\text{some power}}H\{n-1\}^{\text{some power}}.
$$ Indeed $H\{cN\}(a)$ is non-zero since $H\{n-1\}(a) \neq 0$ and since $H\{p-1\}=(\lam-1)^{p-1}$ (\bref{pMinusOneHass}) while $a=1$ is not a root of $H\{n\}$ (\bref{Hsimple}). This completes the proof. 
\end{proof}

As promised, we deal with the $p=2$ case:
\begin{prop} Let $K$ be a field of prime characteristic $p=2$. Fix a polynomial:
$$
f_a = xy(x+y)(x+a y),\, a\in K-\{0,1\}
$$ Then $ FT(f_a) = \frac{1}{2}$.
\end{prop}
\begin{proof}
Note that \pref{keyObPlcG} holds but we cannot replicate the same proof as in \bref{thmTwoLSimp} as, for example, $N=(1/2)(p^e\pm1)$ is not an integer. We need to use different $N$'s. For the upper bound, use $N=\frac{1}{2}p^e$ (we intentionally do not plug in $p=2$ for clarity). Then $f_a^N$ is in $(x,y)^{[p^e]}$ thus $N/p^e=1/2$ is an upper bound. 

As for the lower bound, use $N=\frac{1}{2}(p^e-2)$. Then $f_a^N$ has a monomial $x^{2N}y^{2N}$ with $2N=p^e-2<p^e$. As long as $a$ is not a root of $H\{N\}$, $N/p^e=(1/2)(1-2/p^e)$ is a lower bound, which approaches to $1/2$ as $e\to \infty$. Notice that:
$$
N = \frac{1}{2}(p^e-2) = p^{e-1} -1 = 1+p+...+p^{e-2}. 
$$
So due to \bref{hasseLucas}
$$
H\{N\} = H\{1\}^{\text{some power}} = (1+\lam)^{\text{some power}}.
$$
Since $a\neq 1$, $H\{N\}(a)\neq 0$ and we are done. 
\end{proof}

\begin{disc}\labelt{restForms} For completeness, let us present all possible values of the $F$-pure threshold of a bivariate degree four homogeneous polynomial with four roots, not necessarily distinct.
Consider again these five forms:
$$
x^4, x^3y, x^2y^2, x^2y(x+y), xy(x+y)(x+\lam y) \text{ with } \lam \in K-\{0,1\},
$$
Indeed suffices to compute $FT(f)$ for each of these cases. The monomial cases are straightforward; it is easy to show that $FT(x_1^{a_1}x_2^{a_2}\cdots x_t^{a_t})$ is $(\max(a_1,...,a_t))^{-1}$ (\cite[Example 3.10]{KSbasic}). 
The $f=x^2y(x+y)$ case is treated in \cite{HerBin} as it is a binomial, and it is easy to see that the $F$-pure threshold in this case is $\frac{1}{2}$.
The last case is the subject of \bref{mainThm}.
\end{disc}

\section{Conclusions for Legendre polynomials}\labelt{corSec}
 
An immediate consequence of \bref{mainThm} is the following conclusion (see the equivalent result for Deuring polynomials in \cite{BrillMor04}).
\begin{cor}\labelt{rootCor} \mbox{}\\
\begin{enumerate}
\item Fix a prime $p>2$, and let $n=\frac{p-1}{2}$. If $a\in \overline{\FF_p}-\{0,1\}$ is a root of $H\{n\}$, then so are:
\begin{equation}\labelt{allroots}
(a)^{\pm1}, (1-a)^{\pm1},\left(\frac{a}{a-1}\right)^{\pm1}.
\end{equation}

\item Fix a prime $p>2$, a field $K$ of characteristic $p$ and let $n=\frac{p-1}{2}$. If $b \in K-\{\pm 1\}$ is a root of the Legendre polynomial of degree $n$, $P_n(x) \in K[x]$, then also:
$$
\pm b, \pm \frac{3+b}{-1+b}, \pm \frac{3-b}{1+b}.
$$
\end{enumerate}
\end{cor}

\bref{thmTwoL} gives rise to another corollary; (this statement is also known in the context of Legendre polynomials).
\begin{cor}\labelt{secondCor} Fix a prime $p>2$. Let choose $b,c\in \ZZ_{>0}$ \st $p \equiv 1 \pmod{(b+c)}$. Let $a\in \overline{\FF_p}-\{0,1\}$, then:
$$
H\left\{ \frac{b}{b+c}(p-1)\right\}(a)=0 \iff H\left\{ \frac{c}{b+c}(p-1)\right\}(a)=0.
$$
\end{cor}
\bigskip

The following discussion provides new proofs to both corollaries. Let $K=\overline{K}$ and consider a degree four homogeneous polynomial $f\in K[x,y]$ with distinct roots $(z_1,z_2,z_3,z_4)$ over $\PP_{{K}}^1$. The linear change of variables needed to get the form 
\begin{equation}\labelt{fixedForm}
f_a = xy(x+y)(x+a y) \text{ with } a \in K-\{0,1\}
\end{equation}  
sends:
$$
(z_1,z_2,z_3,z_4) \mapsto (0,\infty, -1, -a),
$$
and a quick computation reveals that $a$ is the cross-ratio:
$$
a = \frac{z_4-z_1}{z_4-z_2}\frac{z_3-z_2}{z_3-z_1}
$$
Since the roots are all distinct, $a$ is not $0,1$ or $\infty$. Notice that $a$ depends on the order we had chosen for the roots. Considering all possible orders, we can get the same form (\ref{fixedForm}) only with one of the following: $a, 1/a, 1-a, 1/(1-a), a/(a-1), (a-1)/a$. This can be done using a linear change of variables, thus the value of the $F$-pure threshold is preserved. With the notation from \pref{fixedForm}, we conclude that:
$$
FT\left(f_a\right) = FT\left(f_{1/a}\right)=FT\left(f_{1-a}\right)=FT\left(f_{1/(1-a)}\right)=FT\left(f_{a/(a-1)}\right)=FT\left(f_{(a-1)/a}\right),
$$
However, the conclusion of \bref{mainThm} is independent of the implicit order we had chosen for the roots. This geometrical insight reveals the interesting property of the roots of $H\left\{\frac{p-1}{2}\right\}$ over $\overline{\FF_p}$ mentioned in the first statement of \bref{rootCor}. Note that $H\{n\}(a)=0\iff H\{n\}(1/a)=0$ is expected due to the symmetry in \bref{deuDef}:
 \begin{equation}\labelt{deurSym}
 H\{n\}(\lam) = \lam^nH\{n\}(1/\lam),
 \end{equation} but the inferring on the rest of the roots in \pref{allroots} is not at all trivial. The second statement of \bref{rootCor} is obtained by rewriting the first statement using \pref{DeurToLeg}.

A similar analysis, performed in the case of \bref{thmTwoL}, gives us \bref{secondCor}: Consider a homogeneous polynomial over $K[x,y]$, $K=\overline{K}$, with 4 distinct (ordered) roots $(z_1,z_2,z_3,z_4)$ over $\PP_K^1$ of multiplicities $b,b,c,c$ respectively. After a linear change of variables the polynomial adopts the form:
\begin{equation}\labelt{thmTwoEq}
f_a = x^by^b(x+y)^c(x+a y)^c,\, a\in K-\{0,1\}.
\end{equation}
In order to do so, one maps
$$
(z_1,z_2,z_3,z_4) \mapsto (0,\infty, -1, -a),
$$
which yields the same cross-ratio:
$$
a = \frac{z_4-z_1}{z_4-z_2}\frac{z_3-z_2}{z_3-z_1}
$$
Considering the result in \bref{thmTwoL}, it is crucial to notice the value of $FT(f_a)$ is symmetric in $b,c$ but we cannot arbitrarily reorder the roots --- 0 and $\infty$ has to have the same multiplicity to obtain the form (\ref{thmTwoEq}), possibly with $b$ and $c$ interchanged. A computation shows that we can get the same form with $1/a$ instead of $a$, while the other values of the cross-ratio are not allowed when $b\neq c$. However, since we can interchange $b$ and $c$ we get that:
$$
H\left\{ \frac{b}{b+c}(p-1)\right\}(a)=0 \iff H\left\{ \frac{c}{b+c}(p-1)\right\}(a)=0.
$$
This proves \bref{secondCor}. The argument presents a new proof of a known fact in the context of Legendre polynomials. 

\bibliographystyle{amsalpha}
\bibliography{unibib}

\end{document}